\documentclass[11pt]{amsart}

\usepackage[margin=3cm]{geometry}

\usepackage{amssymb}

\usepackage{enumerate}

\usepackage{graphicx}

\makeatletter
\@namedef{subjclassname@2010}{%
  \textup{2010} Mathematics Subject Classification}
\makeatother

\newtheorem{thm}{Theorem}

\newtheorem{lem}{Lemma}
\newtheorem{prop}{Proposition}

\theoremstyle{definition}

\newtheorem{rem}{Remark}

\newtheorem{conj}{Conjecture}

\begin{document}

\title[Sign changes in Mertens' theorems]{Sign changes in Mertens' first and second theorems}

\author[Jeffrey Lay]{Jeffrey P.S. Lay}
\address{Mathematical Sciences Institute, The Australian National University, Canberra ACT 2601, Australia}
\email{jeffrey.lay@anu.edu.au}

\date{}

\begin{abstract}
We show that the functions $\sum_{p\leq x} (\log p)/p - \log x - E$ and $\sum_{p\leq x} 1/p - \log \log x -B$ change sign infinitely often, and that under certain assumptions, they exhibit a strong bias towards positive values. These results build on recent work of Diamond and Pintz~\cite{DiamondPintz2009} and Lamzouri~\cite{Lamzouri2014} concerning oscillation of Mertens' product formula, and answers to the affirmative a question posed by Rosser and Schoenfeld~\cite{RosserSchoenfeld1962}.
\end{abstract}



\maketitle

\section{Introduction}

Mertens' first two theorems concerning the density of the primes can be stated as the asymptotic formulae (see~\cite[Thms 11 \& 12]{Dusart1999} for explicit bounds)
\begin{equation*}
\begin{aligned}
	M_1(x) &:= \sum_{p\leq x} {\log p \over p} - \log x - E = O \bigg( {1 \over \log x} \bigg), \\
	M_2(x) &:= \sum_{p\leq x} {1 \over p} - \log \log x - B = O \bigg( {1 \over \log^2 x} \bigg), \\
\end{aligned}
\end{equation*}
as $x\to \infty$, where, writing $C_0$ for Euler's constant,
\begin{equation*}
	E:=-C_0-\sum_{k=2}^\infty \sum_p {\log p \over p^k} = -1.332\dots, \quad B:= C_0 - \sum_{k=2}^\infty \sum_p {1 \over k p^k} = 0.261\dots.
\end{equation*}

Concerning the signs of $M_1(x)$ and $M_2(x)$, calculations by Rosser and Schoenfeld~\cite[Thms 20 \& 21]{RosserSchoenfeld1962} show that $M_1(x)>0$ and $M_2(x)>0$ for all $1<x\leq 10^8$, and they questioned, by analogy with Littlewood's famous result on $\pi(x)-\text{li}(x)$, whether both inequalities fail for arbitrarily large $x$.

Diamond and Pintz~\cite{DiamondPintz2009} established oscillation in Mertens' product formula, answering an analogous question of Rosser and Schoenfeld. Precisely, they showed that the function
\begin{equation*}
	\sqrt x \left( \prod_{p\leq x} \left( 1 - {1 \over p} \right)^{-1} - e^{C_0} \log x \right) 
\end{equation*} 
attains arbitrarily large positive and negative values as $x\to \infty$. Motivated by their work, we prove that both $M_1(x)$ and $M_2(x)$ change sign infinitely often. Moreover, we provide estimates regarding the growth of their oscillations. 

\begin{thm}
		\label{t1}
For each $i\in \{1,2\}$ the following assertion holds: There exists a function $f_i(x)$ going to infinity as $x\to \infty$ such that
\begin{equation*}
	\liminf_{x\to \infty} {\sqrt x \log^{i-1}x \over f_i(x)} M_i(x) < -1, \quad \limsup_{x\to \infty} {\sqrt x \log^{i-1} x \over f_i(x)} M_i(x) > 1.
\end{equation*}
\end{thm}

\begin{rem}
Our methods show that if the Riemann Hypothesis (RH) is true, we may take $f_i(x)=c_i \log \log \log x$ for some fixed $c_i>0$ (cf.~\S{4} in~\cite{DiamondPintz2009}). Estimating the growth of these functions unconditionally appears to remain a formidable problem.
\end{rem}


\begin{rem}
The main result in this paper is the case $i=1$ of Theorem~\ref{t1}, since oscillation of $M_2(x)$ follows as a simple corollary from oscillation of Mertens' product formula~\cite[Thm~1.1]{DiamondPintz2009}, in view of the asymptotic 
\begin{equation*}
	-\log \bigg( 1 - {1 \over p} \bigg) = \sum_{k=1}^\infty {1 \over k p^k} = {1 \over p} + O \bigg( {1 \over p^2} \bigg).
\end{equation*}
We refer the reader to \S{\ref{m2}} for exact details.
\end{rem}

Due to the nature of oscillation theorems, it is convenient to break the proof of the case $i=1$ of Theorem~\ref{t1} into two cases, the first in which RH is assumed to fail, and the second in which it is assumed to hold. We tackle these individual cases in \S{\ref{m11}} and \S{\ref{m12}}, respectively. 

We investigate in \S{\ref{bias}} why the functions $M_1(x)$ and $M_2(x)$ are biased towards positive values, explaining the observations of Rosser and Schoenfeld. In general, we say that $f(x)$ is \emph{biased} towards values in $S\subset \mathbb R$ if $\delta \big( \{x:f(x)\in S\} \big)>1/2$ for an appropriate notion of density~$\delta$. It turns out (see, for example,~\cite{Wintner1941}) that the \emph{logarithmic density} is the appropriate density to use for oscillation theorems; suffice it to say, the usual density does not exist. We recall its definition: for any $S\subset \mathbb R$, we define
\begin{equation*}
	\underline \delta(S) := \liminf_{X\to \infty} {1 \over \log X} \int_{t\in S \cap [2,X]} {dt \over t}, \quad \overline \delta(S) := \limsup_{X\to \infty} {1 \over \log X} \int_{t\in S \cap [2,X]} {dt \over t}.
\end{equation*}
If $\underline \delta(S)=\overline \delta(S)$, we call the resultant quantity the logarithmic density of $S$, and denote it by~$\delta(S)$. 

We also recall the following conjecture concerning the vertical distribution of the non-trivial zeroes of the Riemann zeta-function $\zeta(s)$.

\begin{conj}[Linear Independence Hypothesis (LI)] 
The set of positive ordinates of the non-trivial zeroes of $\zeta(s)$ is linearly independent over $\mathbb Q$.
\end{conj}

This conjecture encapsulates the widely-held belief that there should not exist any algebraic relations between the non-trivial zeroes of $\zeta(s)$; it also implies that all such zeroes are simple. Analogous statements are expected to hold for generalised $L$-functions.

Rubinstein and Sarnak \cite{RubinsteinSarnak1994} showed that under the assumption of both RH and LI, we have
\begin{equation*}
	\delta(1) := \delta \big( \{x\geq 2: \pi(x) > \text{li}(x)\} \big) = 0.00000026\dots;
\end{equation*}
thus, the difference $\pi(x)-\text{li}(x)$ is highly biased towards negative values. Lamzouri~\cite{Lamzouri2014} recently studied the bias in Mertens' product formula using the framework developed by Rubinstein and Sarnak. He determined that under the assumptions of RH and LI, we have
\begin{equation*}
	\delta \left( \left\{ x\geq 2:\prod_{p\leq x} \left(1-{1 \over p} \right)^{-1} > e^\gamma \log x \right\} \right) = 1- \delta(1) = 0.99999973\dots. 
\end{equation*}
We shall prove that $M_1(x)$ and $M_2(x)$ are both biased towards positive values with logarithmic density $1-\delta(1)$.

\begin{thm}
		\label{t2}
For each $i\in \{1,2\}$ the following assertion holds: Let $\mathcal W_i$ denote the set of real numbers $x\geq 2$ such that $M_i(x)>0$. Then, assuming RH, we have $0<\underline \delta(\mathcal W_i) \leq \overline \delta(\mathcal W_i) <1$. If in addition to RH we assume LI, then in fact $\delta(\mathcal W_i) = 1 - \delta(1)$.
\end{thm}

We shall see that the case $i=2$ follows immediately from the work of Lamzouri~\cite{Lamzouri2014}, owing to the almost identical behaviour between $M_2(x)$ and the logarithmic form of Mertens' product formula. A full proof will be given for the case $i=1$; we follow the argument given in~\cite{Lamzouri2014}.

\subsection*{Acknowledgements}

I would like to thank my supervisor, Tim Trudgian, for his invaluable guidance and encouragement, and for many helpful discussions. This research was partially supported by an Australian Postgraduate Award.

\section{Notation} \label{notation}

As usual, we write $f(x)=O\big(g(x)\big)$ or, equivalently, $f(x)\ll g(x)$, if $|f(x)|\leq c g(x)$ is satisfied for some $c>0$ and all sufficiently large $x$. For oscillation estimates, we say that $f(x)=\Omega_{\pm}\big(g(x)\big)$ if there exists $c'\geq0$ such that both $\liminf_{x\to \infty} f(x)/g(x) <-c'$ and $\limsup_{x\to \infty} f(x)/g(x)>c'$ hold.  

For a complex variable $s=\sigma+it$, $\Re s$ and $\Im s$ will denote, respectively, the real and imaginary parts of $s$. The letter $p$ will always represent a prime number, and we use $\rho=\beta+i\gamma$ to denote a non-trivial zero of $\zeta(s)$.

Finally, since we will be using some probability theory, we write $\mathbb P$ for probability and $\mathbb E$ for expectation.

\section{Oscillation of $M_1(x)$: the non-RH case} \label{m11}

The first step is to replace the terms in $M_1(x)$ involving sums over primes with an appropriate Stieltjes integral.
\begin{lem}
		\label{l6}
We have
\begin{equation*}
	M_1(x) = \int_1^x {d\psi(t) \over t} - \log x + C_0 + O \bigg( {1 \over \sqrt x} \bigg),
\end{equation*}
where, as usual, $\psi(x):=\sum_{n\leq x} \Lambda(n) = \sum_{p^k\leq x} \log p$.
\end{lem}

\begin{proof}
We observe that
\begin{equation*}
\begin{aligned}
	\int_1^x {d \psi(t) \over t} = \sum_{n\leq x} {\Lambda(n) \over n} &= \sum_{p\leq x} {\log p \over p} + \sum_{\substack{p^k\leq x \\ k\geq 2}} {\log p \over p^k} \\
	&= \sum_{p\leq x} {\log p \over p} + \sum_{k=2}^\infty \sum_p {\log p \over p^k} - \Bigg( \sum_{\substack{p\leq x \\ p^k>x}} {\log p \over p^k} + \sum_{\substack{p>x \\ k\geq 2}} {\log p \over p^k} \Bigg),
\end{aligned}
\end{equation*}
so it remains to estimate the term in brackets. Using the well-known estimate $\theta(x):=\sum_{p\leq x} \log p \ll x$, we see that
\begin{equation*}
	\sum_{\substack{p>x \\ k\geq 2}} {\log p \over p^k} \ll \sum_{p>x} {\log p \over p^2} = \int_x^\infty {d\theta(t) \over t^2} \ll {1 \over x} + \int_x^\infty {\theta(t) \over t^3} \, dt \ll {1 \over x}.
\end{equation*}
For the remaining sum, we use the estimate $\pi(x) \ll x/\log x$ to obtain
\begin{equation*}
	\sum_{\substack{p\leq x \\ p^k>x}} {\log p \over p^k} \ll \sum_{\sqrt x<p\leq x} {\log p \over p^2} + \sum_{p\leq \sqrt x} {\log x \over x} \ll {1 \over \sqrt x}.
\end{equation*}
We conclude that
\begin{equation}
		\label{ee}
	\sum_{p\leq x} {\log p \over p} + \sum_{k=2}^\infty \sum_p {\log p \over p^k} = \int_1^x {d\psi(t) \over t} + O \bigg( {1 \over \sqrt x} \bigg),
\end{equation}
from which the result follows.
\end{proof}

Now set
\begin{equation*}
	\mathcal U(x) := \int_1^x {d\psi(t) \over t} - \log x + C_0, \quad \mathcal V(x):= {1 \over \sqrt x}.
\end{equation*}
The theorem clearly follows if we can show that for any fixed $K\in \mathbb R$, the function $\mathcal U(x)+K\mathcal V(x)$ changes sign infinitely often. This is achieved through an application of the following famous theorem of Landau (see~\cite[Thm~H]{InghamDistPrimes}).

\begin{thm}[Landau's oscillation theorem] 
		\label{tl}
Suppose $f(x)$ is of constant sign for all sufficiently large $x$. Then the real point $s=\sigma_0$ of the line of convergence of the Dirichlet integral $\int_1^\infty x^{-s} f(x) \, dx$ is a singularity of the function represented by the integral.
\end{thm}

This approach naturally leads to a consideration of the Mellin transforms of $\mathcal{U}(x)$ and $\mathcal{V}(x)$. Recall the well-known identity~\cite[Eq.~(17)]{InghamDistPrimes}
\begin{equation}
		\label{eee}
	-{\zeta' \over \zeta}(s) = s\int_1^\infty x^{-s-1} \psi(x) \, dx = \int_1^\infty x^{-s} \, d\psi(x), \quad \Re s>1,
\end{equation}
whence it follows from a change of variables and integration by parts that
\begin{equation*}
	-{\zeta' \over \zeta}(s+1) = \int_1^\infty x^{-s} \, {d\psi(x) \over x} = s \int_1^\infty x^{-s-1} \int_1^x {d\psi(t) \over t} \, dx, \quad \Re s>0.
\end{equation*}
Moreover, we have from elementary means
\begin{equation*}
	\int_1^\infty x^{-s-1} \log x \, dx = {1 \over s^2}, \quad \Re s>0,
\end{equation*}
and
\begin{equation}
		\label{e5}
	\int_1^\infty x^{-s-1} \, dx = {1 \over s}, \quad \Re s>0.
\end{equation}
These give us the Mellin transforms
\begin{equation*}
	\widehat{\mathcal U}(s) := \int_1^\infty x^{-s-1} \mathcal U(x) \, dx = -{1 \over s} {\zeta' \over \zeta}(s+1) - {1 \over s^2} + {1 \over s}C_0, \quad \Re s>0,
\end{equation*}
and, replacing $s$ with $s+1/2$ in (\ref{e5}),
\begin{equation*}
	\widehat{ \mathcal V}(s) := \int_1^\infty x^{-s-1} \mathcal V(x) \, dx = \int_1^\infty x^{-(s+1/2)-1} \, dx = {2 \over 1+2s}, \quad \Re s>-{1 \over 2}.
\end{equation*}
We now investigate the analytic behaviour of the point of convergence of the Mellin transform $\widehat{\mathcal U}(s)+K\widehat{\mathcal V}(s)$.

Consider the explicit formula~\cite[Cor.\ 10.14]{MontgomeryVaughan}
\begin{equation}
		\label{e9}
	-{\zeta' \over \zeta}(s) = {1 \over 2}C_0+1-\log 2\pi + {1 \over s-1} + {1 \over 2}{\Gamma' \over \Gamma} \bigg( {s \over 2}+1 \bigg)-\sum_\rho \left( {1 \over s-\rho}+{1 \over \rho} \right),
\end{equation}
whence 
\begin{equation}
		\label{e7}
\begin{aligned}
	-{1 \over s}{\zeta' \over \zeta}(s+1)&= {1 \over s} \bigg( {1 \over 2}C_0+1-\log 2\pi \bigg) + {1 \over s^2} \\
	&\quad \quad \quad + {1 \over 2s} {\Gamma' \over \Gamma} \bigg( {s+1 \over 2}+1 \bigg) - {1 \over s} \sum_\rho \left( {1 \over s+1-\rho}+{1 \over \rho} \right).
\end{aligned}
\end{equation}
It remains to classify the simple poles of the last two terms on the right-hand side of~(\ref{e7}).

\begin{lem}
		\label{l3}
We have
\begin{equation*}
	-{1 \over s} \sum_\rho \left( {1 \over s+1-\rho}+{1 \over \rho} \right) = {1 \over s} \big( -C_0-2+\log 4\pi \big)+F(s),
\end{equation*}
where $F(s)$ is some function regular for $\Re s>0$.
\end{lem}

\begin{proof}
This follows from the identity \cite[Eq.\ (10.30)]{MontgomeryVaughan}
\begin{equation}
		\label{e12}
	-\sum_\rho \left( {1 \over 1-\rho}+{1 \over \rho} \right) = -C_0 - 2 + \log 4\pi,
\end{equation}
whence
\begin{equation*}
	\underset{s=0}{\text{Res}} \left[ -{1 \over s}\sum_\rho \left( {1 \over s+1-\rho} + {1 \over \rho} \right) \right] = -C_0 - 2 + \log 4\pi.
\end{equation*}
\end{proof}

\begin{lem}
		\label{l4}
We have
\begin{equation*}
	{1 \over 2s} {\Gamma' \over \Gamma} \bigg( {s+1 \over 2}+1 \bigg) = {1 \over s} \bigg( -{1 \over 2}C_0 - \log 2 + 1 \bigg) + G(s),
\end{equation*}		
where $G(s)$ is some function regular for $\Re s>0$.
\end{lem}

\begin{proof}
Logarithmically differentiating Legendre's duplication formula~\cite[Eq.~(C.9)]{MontgomeryVaughan}
yields the functional equation
\begin{equation*}
	{\Gamma' \over \Gamma} \bigg({1 \over 2} + s \bigg) = -{\Gamma' \over \Gamma}(s) - 2 \log 2 + 2 {\Gamma' \over \Gamma}(2s).
\end{equation*}
Using the fact that for $n\in \mathbb N$ we have $\Gamma'(n+1)=n! \times \big(-C_0+\sum_{k=1}^n 1/k \big)$, and recalling the identity $-\Gamma'(1)=C_0$, we deduce that
\begin{equation}
		\label{e13}
	{\Gamma' \over \Gamma} \bigg( {1 \over 2}+1 \bigg) = -{\Gamma' \over \Gamma}(1) - 2 \log 2 + 2 {\Gamma' \over \Gamma}(2) = -C_0 - 2 \log 2 + 2,
\end{equation}
whence
\begin{equation*}
	\underset{s=0}{\text{Res}} \left[ {1 \over 2s} {\Gamma' \over \Gamma} \bigg( {s+1 \over 2} + 1 \bigg) \right] = -{1 \over 2}C_0 - \log 2 + 1.
\end{equation*}
\end{proof}

Combining Lemmas \ref{l3} and \ref{l4} with equation (\ref{e7}) gives us the formula
\begin{equation*}
	-{1 \over s} {\zeta' \over \zeta} (s+1) = -{1 \over s}C_0 + {1 \over s^2} + F(s) + G(s),
\end{equation*}
whence
\begin{equation}
		\label{e8}
	\widehat{\mathcal U}(s) + K\widehat{\mathcal V}(s) = F(s) + G(s) + K {2 \over 1+2s}.
\end{equation}

To conclude the proof of the theorem, fix $K$ (positive or negative) and suppose RH is false. Then $\zeta'(s+1)/\zeta(s+1)$ has a singularity at a complex point $s_0$ with $\Re s_0>-1/2$, so the abscissa of convergence of the Mellin transform 
\begin{equation*}
	\widehat{\mathcal U}(s) + K\widehat{\mathcal V}(s) = -{1 \over s} {\zeta' \over \zeta}(s+1)-{1 \over s^2}+{1 \over s}C_0 + K{2 \over 1+2s}
\end{equation*}
is at least $-1/2$. But (\ref{e8}) shows that the possible singularity at $s=0$ is removable, so we conclude that the point of convergence of the Mellin transform is a regular point. It follows from Theorem~\ref{tl} that $\mathcal{U}(x)+K\mathcal{V}(x)$ changes sign infinitely often.

\section{Oscillation of $M_1(x)$: the RH case} \label{m12}

We start with a formula that relates $M_1(x)$ to the error term in the prime number theorem. 

\begin{lem}
		\label{l1}
We have (unconditionally)
\begin{equation}
		\label{e2}
	M_1(x) = {\psi(x) - x \over x} - \int_x^\infty {\psi(t) - t \over t^2} \, dt + O \bigg( {1 \over \sqrt x} \bigg).
\end{equation}	
\end{lem}

\begin{proof}
Integration by parts yields
\begin{equation*}
\begin{aligned}
	\int_1^x {d\psi(t) \over t} &= {\psi(x) \over x} + \int_1^x {\psi(t) \over t^2} \, dt \\
	&= \log x + {\psi(x)-x \over x} + 1 + \int_1^x {\psi(t)-t \over t^2} \, dt \\
	&= \log x + {\psi(x)-x \over x} + 1 + \int_1^\infty {\psi(t)-t \over t^2} \, dt - \int_x^\infty {\psi(t)-t \over t^2} \, dt,
\end{aligned}
\end{equation*}
making use of the fact that $\lim_{x\to \infty} \int_1^x \big(\psi(t)-t \big)/t^2 \, dt \ll 1$, which can be seen via a simple application of the prime number theorem. We conclude from (\ref{ee}) that
\begin{equation*}
	M_1(x) - C_0 - \bigg(1+\int_1^\infty {\psi(t)-t \over t^2} \, dt \bigg) =  {\psi(x)-x \over x} - \int_x^\infty {\psi(t)-t \over t^2} \, dt + O \bigg({1 \over \sqrt x} \bigg).
\end{equation*}

It remains to assign an explicit value to the constant term $1+\int_1^\infty \big(\psi(t)-t\big)/t^2 \, dt$. We first observe that
\begin{equation*}
	1+\int_1^\infty {\psi(t) - t \over t^2} \, dt = 1+\lim_{s\to 1} \int_1^\infty t^{-s-1} \big( \psi(t)-t \big) \, dt,
\end{equation*}
which, appealing to (\ref{eee}) and (\ref{e5}), is equal to
\begin{equation*}
	1+ \lim_{s\to 1} \left[-{1 \over s} {\zeta' \over \zeta}(s) - {1 \over s-1} \right] = \lim_{s\to 1} \left[ - {1 \over s} {\zeta' \over \zeta}(s) + {1 \over s} - {1 \over s-1} \right] = \lim_{s\to 1} \left[ -{1 \over s} {\zeta' \over \zeta}(s) - {1 \over s(s-1)} \right].
\end{equation*}
We now apply (\ref{e9}), (\ref{e12}), and (\ref{e13}) to deduce that the limit attains the value $-C_0$, as desired. 
\end{proof}

Using the famous oscillation result \cite[Thm~34]{InghamDistPrimes} of $\psi(x)-x$, we obtain
\begin{equation}
		\label{e4}
	{\psi(x)-x \over x} = \Omega_{\pm} \bigg( {\log \log \log x \over \sqrt x} \bigg),
\end{equation}
which immediately gives the desired estimate for the first term on the right-hand side of (\ref{e2}). The theorem follows if we can show that the integral is sufficiently small. 

To achieve this, we invoke the following powerful result of Cram\'er~\cite[Thm IV]{Cramer1921} concerning the average order of the error term in the prime number theorem. This enables us to save a logarithmic factor that we would otherwise have to deal with using point-wise estimates.

\begin{thm}[Cram\'er]
If RH is true, then
\begin{equation*}
	{1 \over x} \int_1^x \left| {\psi(t)-t \over \sqrt t} \right| \, dt \ll 1.
\end{equation*}
\end{thm}

From this, we see that
\begin{equation*}
	{1 \over \sqrt{2x}} \int_x^{2x} \big| \psi(t)-t \big| \, dt + \int_1^x \left| {\psi(t)-t \over \sqrt t} \right| \, dt \leq \int_1^{2x} \left| \psi(t)-t \over \sqrt t \right| \, dt \ll x,
\end{equation*}
whence
\begin{equation}
		\label{e3}
	\int_x^{2x} \big| \psi(t)-t \big| \, dt \ll x \sqrt x.
\end{equation}
The strategy is to use the estimate (\ref{e3}) to bound the integral in (\ref{e2}) using dyadic interval estimates. 

Using (\ref{e3}), we have for all non-negative integers $k$
\begin{equation*}
	\int_{2^k x}^{2^{k+1}x} {\psi(t)-t \over t^2} \, dt \leq {1 \over 2^{2k} x^2} \int_{2^k x}^{2^{k+1}x} \big| \psi(t)-t \big| \, dt \ll \left({1 \over {\sqrt 2}} \right)^k {1 \over \sqrt x}.
\end{equation*}
Thus, we obtain the estimate
\begin{equation*}
	\int_x^\infty {\psi(t)-t \over t^2} = \sum_{k=0}^\infty \int_{2^k x}^{2^{k+1} x} {\psi(t)-t \over t^2} \, dt \ll {1 \over \sqrt x} \sum_{k=0}^\infty \left( {1 \over \sqrt 2} \right)^k \ll {1 \over \sqrt x}.
\end{equation*}
We see that the integral is smaller than the oscillation term (\ref{e4}) when $x\to \infty$, so we conclude from (\ref{e2}) that 
\begin{equation*}
	M_1(x) = \Omega_{\pm} \bigg( {\log \log \log x \over \sqrt x} \bigg).
\end{equation*}

\section{Oscillation of $M_2(x)$} \label{m2}

The main result of Diamond and Pintz~\cite{DiamondPintz2009} towards establishing sign changes of $M_3(x)$ is the following oscillation estimate:

\begin{thm}[Diamond and Pintz]
		\label{p100}
There exists a function $f_3(x)$ going to infinity as $x\to \infty$ such that
\begin{equation*}
	-\sum_{p\leq x} \log \left( 1 - {1 \over p} \right) - \log \log x - C_0 = \Omega_{\pm} \bigg( {f_3(x) \over \sqrt x \log x} \bigg).
\end{equation*}
In particular, we may take $f_3(x)=\log \log \log x$ assuming the truth of RH. 
\end{thm}

Thus, the case $i=2$ of Theorem~\ref{t1} follows immediately upon showing

\begin{lem}
		\label{l20}
We have
\begin{equation}
		\label{e700}
	M_2(x) = -\sum_{p\leq x} \log \left( 1-{1 \over p} \right) - \log \log x - C_0 + O \bigg( {1 \over x} \bigg).
\end{equation}
\end{lem}

\begin{proof}
Taking the Taylor expansion of the logarithmic term yields
\begin{equation*}
	-\sum_{p\leq x} \log \left(1 - {1 \over p} \right) = \sum_{p\leq x} \sum_{k=1}^\infty {1 \over k p^k} = \sum_{p\leq x} {1 \over p} + \sum_{k=2}^\infty \sum_p {1 \over k p^k} - \sum_{\substack{p>x \\ k\geq 2}} {1 \over k p^k},
\end{equation*}
so it remains to show that the last sum is $O\big(x^{-1} \big)$. But this follows readily from the generous estimate
\begin{equation*}
	\sum_{\substack{p>x \\ k\geq 2}} {1 \over kp^k} \ll \sum_{p>x} {1 \over p^2} \ll \int_x^\infty {dt \over t^2} \ll {1 \over x}.
\end{equation*}
\end{proof}

\section{Investigating the bias} \label{bias}

We begin by listing some of the main results in~\cite{Lamzouri2014}.

\begin{prop}[Corollary 2.2 in~\cite{Lamzouri2014}]
		\label{p101}
Assuming RH, we have
\begin{equation}
		\label{e102}
\begin{aligned}
	&\sqrt x \log x \left( -\sum_{p\leq x} \log \left( 1- {1 \over p} \right) - \log \log x - C_0 \right) \\
	& \quad \quad \quad \quad \quad \quad = 1 + 2 \Re \sum_{0<\gamma\leq T} {x^{i\gamma} \over -{1 \over 2}+i\gamma} + O \bigg( {\sqrt x \log^2 (xT) \over T} + {1 \over \log x} \bigg).
\end{aligned}
\end{equation}
\end{prop}

\begin{prop}[See \S{4} in~\cite{Lamzouri2014}]
		\label{p2000}
Let $\tilde{\mathcal W}$ denote the set of real numbers $x\geq 2$ such that $\prod_{p\leq x} ( 1- 1/p )^{-1} > e^\gamma \log x$, and let $\tilde Z$ denote the random variable
\begin{equation*}
	\tilde Z := 1 + 2\Re \sum_{\gamma>0} {\tilde X(\gamma) \over \sqrt{{1 \over 4}+\gamma^2}},
\end{equation*}
where $\tilde X(\gamma)$ is a sequence of independent random variables indexed by the positive imaginary parts of the non-trivial zeroes of $\zeta(s)$. Then, assuming RH and LI, we have 
\begin{equation*}
	\delta(\tilde{\mathcal W})=\mathbb P\big[\tilde Z>0\big]=1-\delta(1).
\end{equation*}
\end{prop}

In fact, we can see straight away why the case $i=2$ of Theorem~\ref{t2} follows from the work of Lamzouri. Combining (\ref{e700}) and (\ref{e102}), we deduce that
\begin{equation*}
\begin{aligned}
	&\sqrt x \log x \left( \sum_{p\leq x} {1 \over p} - \log \log x - B \right) \\
	& \quad \quad \quad \quad \quad \quad = 1 + 2 \Re \sum_{0<\gamma\leq T} {x^{i\gamma} \over -{1 \over 2}+i\gamma} + O \bigg( {\sqrt x \log^2 (xT) \over T} + {1 \over \log x} \bigg),
\end{aligned}
\end{equation*}
so the explicit formula for $M_2(x)$ in terms of the non-trivial zeroes of the Riemann zeta-function is identical to that of Mertens' product formula (\ref{e102}) (up to small error). The rest of this section is thus devoted to proving the case $i=1$: we give full details of this proof, which follows the method of Lamzouri. 

Recall that our goal is to measure the logarithmic density of the set
\begin{equation*}
	\mathcal W_1 = \left\{ x\geq 2: \sum_{p\leq x} {\log p \over p} > \log x + E \right\}.
\end{equation*}
To achieve this, define
\begin{equation*}
	\mathcal E(x) := \sqrt x \left( \sum_{p\leq x} {\log p \over p} - \log x - E \right) = \sqrt x M_1(x),
\end{equation*}
and note that $x\in \mathcal W_1$ if, and only if, $\mathcal E(x)>0$. Our main result is the following formula that explicitly relates $\mathcal E(x)$ to the non-trivial zeroes of~$\zeta(s)$.

\begin{prop}
		\label{l5}
For all $x,T\geq 5$ we have
\begin{equation}
		\label{e16}
	\mathcal E(x) = 1 - \sum_{|\gamma|\leq T} {x^{\rho-1/2} \over \rho-1} + O \bigg( {\sqrt x \log^2(xT) \over T}+{1 \over \log x} \bigg).
\end{equation}
\end{prop}

\begin{proof}
Recall from the proof of Lemma \ref{l6} that
\begin{equation*}
	\sum_{p\leq x} {\log p \over p} + \sum_{k=2}^\infty \sum_p {\log p \over p^k} = \sum_{n\leq x} {\Lambda(n) \over n} + \sum_{\substack{p\leq x \\ p^k>x}} {\log p \over p^k} + O\bigg( {1 \over x} \bigg).
\end{equation*}
We require a sharp estimate for the last sum on the right-hand side. First note that
\begin{equation*}
	\sum_{\substack{p\leq x \\ p^k>x}} {\log p \over p^k} = \sum_{\sqrt x < p \leq x} {\log p \over p^2} + O \bigg( {\log x \over x^{2/3}} \bigg),
\end{equation*}
where the contribution from prime powers $p^k$ with $k\geq 3$ was estimated trivially. For the sum over squares of primes, it suffices to use the classical prime number theorem estimate $\theta(x)=x+O\big( x \exp(-c\sqrt{\log x}) \big)$ to obtain
\begin{equation*}
	\sum_{\sqrt x<p\leq x} {\log p \over p^2} = \int_{\sqrt x}^x {d\theta(t) \over t^2} = {1 \over \sqrt x} + O\bigg( {e^{-\sqrt{\log x}} \over \sqrt x} \bigg) = {1 \over \sqrt x} + O\bigg( {1 \over \sqrt x \log x} \bigg),
\end{equation*}
where the last error term was chosen for convenience. Combining the above estimates, we conclude that
\begin{equation}
		\label{e10}
	\sum_{p\leq x} {\log p \over p} + \sum_{k=2}^\infty \sum_p {\log p \over p^k} = \sum_{n\leq x} {\Lambda(n) \over n} + {1 \over \sqrt x} + O \bigg( {1 \over \sqrt x \log x} \bigg).
\end{equation}

We now introduce an explicit formula for the weighted sum of the von Mangoldt function. Lamzouri~\cite[Lem.\ 2.4]{Lamzouri2014} showed that for $\alpha>1$ and $x,T\geq 5$, we have
\begin{equation}
		\label{e500}
\begin{aligned}
	&\sum_{n\leq x} {\Lambda(n) \over n^\alpha} = -{\zeta' \over \zeta}(\alpha) + {x^{1-\alpha} \over 1-\alpha} - \sum_{|\gamma|\leq T} {x^{\rho-\alpha} \over \rho-\alpha} \\
	& \quad \quad \quad + O \Bigg( x^{-\alpha}\log x + {x^{1-\alpha} \over T} \left(4^\alpha + \log^2 x + {\log^2 T \over \log x} \right) + {1 \over T} \sum_{n=1}^\infty {\Lambda(n) \over n^{\alpha+1/\log x}} \Bigg).
\end{aligned}
\end{equation}
Since
\begin{equation*}
	\sum_{n=1}^\infty {\Lambda(n) \over n^{1+1/\log x}} = - {\zeta' \over \zeta} \bigg(1+{1 \over \log x} \bigg) = \log x + O(1), \quad x\to \infty,
\end{equation*}
we therefore obtain, taking the limit $\alpha\to 1^+$ in (\ref{e500}),
\begin{equation}
			\label{e11}
\begin{aligned}
		\sum_{n\leq x} {\Lambda(n) \over n} &= \lim_{\alpha \to 1^+} \left( -{\zeta' \over \zeta}(\alpha) + {x^{1-\alpha} \over 1-\alpha} \right) - \sum_{|\gamma| \leq T} {x^{\rho-1} \over \rho-1} \\
		&\quad \quad \quad \quad \quad \quad \quad \quad \quad + O \bigg( { \log x \over x} + {\log^2 x \over T} + {\log^2 T \over T \log^2 x} \bigg).
\end{aligned}
\end{equation}
To evaluate the limit term in (\ref{e11}), we compute the Laurent series
\begin{equation*}
	{x^{1-\alpha} \over 1-\alpha} = \sum_{k=-1}^\infty {(1-\alpha)^k \log^{k+1} x \over (k+1)!} = {1 \over 1-\alpha} + \log x + {1 \over 2} (1-\alpha) \log^2 x +\cdots,
\end{equation*}
which, together with equations (\ref{e9}), (\ref{e12}), and (\ref{e13}), gives us
\begin{align*}
	\lim_{\alpha\to 1^+} \left( -{\zeta' \over \zeta}(\alpha) + {x^{1-\alpha} \over 1-\alpha} \right) &= {1 \over 2}C_0+1-\log 2\pi + {1 \over 2}{\Gamma' \over \Gamma}\bigg( {1 \over 2}+1 \bigg) \\
	&\quad \quad \quad \quad \quad - \sum_\rho \left( {1 \over 1-\rho}+{1 \over \rho} \right) + \log x \\
	&=-C_0 + \log x.
\end{align*}
Combining (\ref{e10}) and (\ref{e11}), we conclude that
\begin{equation*}
\begin{aligned}
	&\sum_{p\leq x} {\log p \over p} - \log x + C_0 + \sum_{k=2}^\infty \sum_p {\log p \over p^k} \\
	&\quad \quad \quad \quad \quad \quad = {1 \over \sqrt x} - \sum_{|\gamma|\leq T} {x^{\rho-1} \over \rho-1} + O \bigg( {\log^2 (xT) \over T}+{1 \over \sqrt x \log x} \bigg),
\end{aligned}
\end{equation*}
and multiplying through by $\sqrt x$ gives the result.
\end{proof}

\begin{rem}
It is immediately clear from (\ref{e16}) that the constant $1$ is responsible for the positive bias of $M_1(x)$.
\end{rem}

\begin{lem}
		\label{l7}
Assuming RH, we have
\begin{equation}
		\label{e14}
	\mathcal E(x) = 1-2\Re \sum_{0<\gamma\leq T} {x^{i\gamma} \over -{1 \over 2}+i\gamma} + O \bigg( {\sqrt x \log^2(xT) \over T}+{1 \over \log x} \bigg),
\end{equation}
and, in particular,
\begin{equation}
		\label{e15}
	\mathcal E(x) = -2\sum_{0<\gamma\leq T} {\sin (\gamma \log x) \over \gamma} + O \bigg( 1 + {\sqrt x \log^2(xT) \over T} \bigg).
\end{equation}
\end{lem}

\begin{proof}
Equation (\ref{e14}) follows immediately upon writing $\rho=1/2+i\gamma$ in (\ref{e16}). To deduce (\ref{e15}), we combine (\ref{e14}) with the observation that
\begin{equation*}
	\left| \sum_{0<\gamma\leq T} {x^{i\gamma} \over -{1 \over 2}+i\gamma} - \sum_{0<\gamma\leq T} {x^{i\gamma} \over i\gamma} \right| \ll \sum_{0<\gamma\leq T} {1 \over \gamma^2} \ll 1,
\end{equation*}
where convergence of the last sum follows from the Riemann--von-Mangoldt formula.
\end{proof}


The existence of the upper and lower logarithmic densities is due to the following result from Section 2.2 of~\cite{RubinsteinSarnak1994}.

\begin{prop}[Rubinstein and Sarnak]
		\label{p1}
There exists absolute positive constants $a_1$ and $a_2$ such that for all $\lambda \gg 1$ and $Y$ sufficiently large,
\begin{equation*}
	{1 \over Y} \text{\emph{meas}} \left\{ y\in [2,Y] : \sum_{0<\gamma\leq e^Y} {\sin (\gamma y) \over \gamma} > \lambda \right\} \geq {a_1 \over \exp \big( \exp (a_2 \lambda) \big)},
\end{equation*}
and
\begin{equation*}
	{1 \over Y} \text{\emph{meas}} \left\{ y\in [2,Y] : \sum_{0<\gamma\leq e^Y} {\sin (\gamma y) \over \gamma} < -\lambda \right\} \geq {a_1 \over \exp \big( \exp (a_2 \lambda) \big)}.
\end{equation*}
\end{prop}

We are now ready to prove the first assertion of the theorem; henceforth, assume RH. Substituting $y=\log x$ into (\ref{e15}) gives us
\begin{equation*}
	\mathcal E(e^y) = -2\sum_{0<\gamma\leq T} {\sin(\gamma y) \over \gamma} + O \bigg( 1 + {e^{y/2} (y+\log T)^2 \over T} \bigg),
\end{equation*}
whence we deduce that for all sufficiently large $Y$, there exists $A>0$ such that for all $2\leq y\leq Y$,
\begin{equation*}
	-2\left(\sum_{0<\gamma\leq e^Y} {\sin(\gamma y) \over \gamma} + A \right) < \mathcal E(e^y) < - 2 \left( \sum_{0<\gamma\leq e^Y} {\sin(\gamma y) \over \gamma} - A \right).
\end{equation*}
Using this, we see that $\sum_{0<\gamma\leq e^Y} \sin(\gamma y)/\gamma<-A$ implies $\mathcal E(e^Y)>0$. It follows from Proposition~\ref{p1} that
\begin{equation*}
\begin{aligned}
	{1 \over \log x} \int_{t\in \mathcal W_1\cap[2,x]} {dt \over t} &= {1 \over Y} \text{meas} \, \big\{ y\in [\log 2,Y]: \mathcal E(e^y)>0 \big\} \\
	&\geq {1 \over Y} \text{meas} \left\{ y\in[2,Y]: \sum_{0<\gamma\leq e^Y} {\sin(\gamma y) \over \gamma} < -A \right\} \\
	&\geq {1 \over 2} {a_1 \over \exp\big( \exp(a_2 A) \big)},
\end{aligned}
\end{equation*}
say, if $Y$ is large enough. Hence, we deduce that
\begin{equation*}
	\underline{\delta}(\mathcal W_1) \geq {1 \over 2} {a_1 \over \exp \big( \exp(a_2 A) \big)} > 0.
\end{equation*}	
By a similar argument, we see that $\mathcal E(e^Y)>0$ implies $\sum_{0<\gamma\leq e^Y} \sin(\gamma y)/\gamma<A$, whence
\begin{equation*}
\begin{aligned}
	{1 \over \log x} \int_{t\in \mathcal W_1\cap[2,x]} {dt \over t} &\leq {1 \over Y} \text{meas} \left\{ y\in[2,Y] : \sum_{0<\gamma\leq e^Y} {\sin(\gamma y) \over \gamma} < A \right\} + O \bigg( {1 \over Y} \bigg) \\
	&\leq 1 - {1 \over 2} {a_1 \over \exp \big( \exp(a_2 A) \big)},
\end{aligned}
\end{equation*}
say, from which we conclude that $\overline{\delta}(\mathcal W_1)<1$. 

It remains to prove that under the additional assumption of LI, the quantities $\underline{\delta}(\mathcal W_1)$ and $\overline{\delta}(\mathcal W_1)$ coincide and attain the value $1-\delta(1)$. 


\begin{prop}
		\label{p10}
Assuming RH, there exists a probability measure $\mu_{\mathcal E}$ on $\mathbb R$ such that for all bounded continuous functions $u:\mathbb R \to \mathbb R$, we have
\begin{equation}
		\label{e1000}
	\lim_{x\to \infty} {1 \over \log x} \int_2^x u\big( \mathcal E(t) \big) \, {dt \over t} = \int_{-\infty}^\infty u(t) \, d\mu_{\mathcal E}.
\end{equation}
If in addition to RH we assume LI, then we have the following explicit formula for the Fourier transform of $\mu_{\mathcal E}$:
\begin{equation}
		\label{e19}
	\widehat \mu_{\mathcal E}(t) = \int_{-\infty}^\infty e^{-it} \, d\mu_{\mathcal E} = e^{-it} \prod_{\gamma>0} J_0 \left( {2t \over \sqrt{{1 \over 4}+\gamma^2}} \right),
\end{equation}
where $J_0(t) := \sum_{k=0}^\infty (-1)^k (k!)^{-2} (t/2)^{2k}$ is the Bessel function of the first kind of order zero.
\end{prop}

\begin{proof}
Set $y=\log x$ in the explicit formula (\ref{e14}), and let $\upsilon(y,T) := e^{y/2} (y+\log T)^2/T + 1/y$
denote the error term. A simple calculation shows that
\begin{equation*}
	\lim_{Y\to \infty} {1 \over Y} \int_{\log 2}^Y \big| \upsilon(y,e^Y) \big|^2 \, dy = 0;
\end{equation*}
thus, the mean square of the error is uniformly small. It follows from the work of Rubinstein and Sarnak~\cite{RubinsteinSarnak1994} and Akbary, Ng, and Shahabi~\cite[Thm~1.2]{AkbaryNgShahabi2014} that $\mathcal E(x)$ is a $B^2$-almost periodic function and thus possesses a \emph{limiting distribution}~(\ref{e1000}). In particular, the Fourier transform (\ref{e19}) was deduced from \cite[Thm~1.9]{AkbaryNgShahabi2014}.
\end{proof}

Note that under LI, the quantities $x^{i\gamma}$ appearing in equation (\ref{e14}) can be viewed as points uniformly distributed on the unit circle. This leads to the following statistical characterisation of the measure $\mu_{\mathcal E}$.

\begin{lem}
		\label{l10}
Assume RH and LI. Let $X(\gamma)$ denote a sequence of random variables indexed by the positive ordinates of the non-trivial zeroes of $\zeta(s)$, and distributed uniformly on the unit circle. Then $\mu_{\mathcal E}$ is the distribution of the random variable 
\begin{equation*}
	Z := 1 - 2 \Re \sum_{\gamma>0} {X(\gamma) \over \sqrt{{1 \over 4}+\gamma^2}}.
\end{equation*}
\end{lem}

\begin{proof}
We see from the definition of $Z$ that
\begin{equation}
		\label{e20}
	\mathbb E \big[ e^{-it Z} \big] = e^{-it} \prod_{\gamma>0} \mathbb E \left[ \exp \left( i {2t \over \sqrt{{1 \over 4}+\gamma^2}} \Re X(\gamma) \right) \right].
\end{equation}
However, we note that for a random variable $X$ uniformly distributed on the unit circle, 
\begin{equation*}
	\mathbb E \big[ e^{i t \Re X} \big] = {1 \over 2\pi} \int_0^{2\pi} e^{i t \cos \theta} \, d\theta = J_0(t),
\end{equation*}
making use of the integral representation of the Bessel function. Hence, the right-hand side of (\ref{e20}) is equal to 
\begin{equation*}
	 e^{-it} \prod_{\gamma>0} J_0 \left( {2t \over \sqrt{{1 \over 4}+\gamma^2}} \right),
\end{equation*}
so we conclude that $\mathbb E\big[e^{-it Z}\big]=\widehat \mu_{\mathcal E}(t)$ by (\ref{e19}). 
\end{proof}

Now observe that $Z$ and $\tilde Z$ have the same distribution, in view of the fact that the $X(\gamma_n)$ are symmetric random variables. Using Proposition~\ref{p2000}, this implies
\begin{equation*}
\mathbb P \big[Z>0 \big] = \mathbb P\big[\tilde Z>0\big] = 1-\delta(1),
\end{equation*}
so the second assertion of the theorem follows upon showing

\begin{lem}
		\label{l11}
Assuming RH and LI, we have $\delta(\mathcal W_1)=\mathbb P\big[Z>0\big]$.
\end{lem}

\begin{proof}
Since $Z$ is the sum of continuous random variables, it follows from Lemma~\ref{l10} that $\mu_{\mathcal E}$ is an absolutely continuous probability distribution. Let $\epsilon>0$, and let $u_1(x)$ and $u_2(x)$ be continuous functions such that
\begin{equation*}
	u_1(x)=\begin{cases} 1 & \text{if } x\geq 0, \\ \in [0,1] & \text{if } x\in (-\epsilon,0), \\ 0 & \text{otherwise}, \end{cases} \quad u_2(x)=\begin{cases} 1 & \text{if } x\geq \epsilon, \\ \in [0,1] & \text{if } x\in (0,\epsilon), \\ 0 & \text{otherwise}. \end{cases}
\end{equation*}
It follows from Proposition \ref{p10} and Lemma \ref{l10} that
\begin{equation*}
	\overline \delta(\mathcal W_1) \leq \lim_{x\to \infty} {1 \over \log x} \int_2^x u_1 \big( \mathcal E(t) \big) \, {dt \over t} = \int_{-\infty}^\infty u_1(t) \, d\mu_{\mathcal E} \leq \mu_{\mathcal E}(-\epsilon,\infty) = \mathbb P\big[Z>0\big]+O(\epsilon),
\end{equation*}
and, using a similar argument,
\begin{equation*}
	\underline \delta(\mathcal W_1) \geq \lim_{x\to \infty} {1 \over \log x} \int_2^x u_2 \big( \mathcal E(t) \big) \, {dt \over t} = \int_{-\infty}^\infty u_2(t) \, d\mu_{\mathcal E} \geq \mu_{\mathcal E}(\epsilon,\infty) = \mathbb P\big[Z>0\big]+O(\epsilon).	
\end{equation*}
The result follows on taking $\epsilon \to 0$.
\end{proof}

\end{document}